\documentclass[Svgc_modified]{Svgc_modified}
\usepackage{amsmath}
\usepackage{graphicx}
\usepackage{enumitem}
\usepackage{subfigure}
\usepackage{stmaryrd}
\usepackage{tikz,xcolor}
\usepackage[utf8]{inputenc}
\usepackage[T1]{fontenc}
\def\1{$\ }
\def\2{\ $}

\def\N{\mbox{\makebox[.2em][l]{I}N}}

\def\cE{{\cal E}}
\def\cH{{\cal H}}

\def\hqed{\hfill\framebox(6,6){\ }}

\def\ul{\underline}

\def\eiE{\cE^{EI}}

\begin{document}
\title{Cycles as edge intersection hypergraphs of $k$-uniform hypergraphs ($k \le 6$) -- a constructive approach }
%$(k \le 6)$
% -- trees, cacti and clique-fusions}
%\subtitle{Insert your subtitle here, if you have.}
\author{Sophie P\"{a}tz\inst{1} \and Martin Sonntag\inst{2}}% etc
%}                     % Do not remove

%running header
\titlerunning{Cycles as edge intersection hypergraphs}
%\authorrunning{1st. author\inst{1} \and 2nd. author\inst{2}}%

\institute{Formerly: Faculty of Mathematics and Computer Science, Technische Universit\"{a}t Bergakademie Freiberg, Pr\"{u}ferstra\ss e 1, 09596 Freiberg, Germany
%\email {teichert@math.uni-luebeck.de}
\and Faculty of Mathematics and Computer Science, Technische Universit\"{a}t Bergakademie Freiberg, Pr\"{u}ferstra\ss e 1, 09596 Freiberg, Germany
%\email{sonntag@tu-freiberg.de}
}
\maketitle
\begin{abstract}
If $\cH=(V,\cE)$ is a hypergraph, its {\it edge intersection hypergraph} $EI(\cH)=(V,\eiE)$ has the edge set $\eiE=\{e_1 \cap e_2 \ |\ e_1, e_2 \in \cE \ \wedge \ e_1 \neq e_2  \ \wedge  \ |e_1 \cap e_2 |\geq2\}$. In the present paper, we consider 4- and 5-uniform hypergraphs $\cH$, respectively, with $EI(\cH) = C_n$.  Our results fill the gap between the 3- and the 6-uniform case considered in \cite{ST}.
\end{abstract}
%
%\receive{xxxxxxxxxxxxxxxxxxxx}
%\finalreceive{yyyyyyyyyyyyyyyyy}\
\section{Introduction and basic definitions}
All hypergraphs $\cH = (V(\cH), \cE(\cH))$ and (undirected) graphs $G=(V(G), E(G))$  considered in the following may have isolated vertices but no multiple edges or loops.

A hypergraph $\cH = (V, \cE)$ is {\em $k$-uniform} if all hyperedges $e \in \cE$ have the cardinality $k$.
Trivially, any 2-uniform hypergraph $\cH$ is a graph.
The {\em degree} $d(v)$ (or $d_{\cH}(v)$) of a vertex $v \in V$ is the number of hyperedges $e \in \cE$ being incident to the vertex $v$. $\cH$ is {\em $r$-regular} if all vertices $v \in V$ have the same degree $r = d(v)$.
%$\cH$ is {\em linear} if any two distinct hyperedges $e, e' \in \cE$ have at most one vertex in common.
In standard terminology we follow Berge \cite{GST5}.

If $\cH=(V,\cE)$ is a hypergraph, its {\it edge intersection hypergraph} $EI(\cH)=(V,\eiE)$ has the edge set $\eiE=\{e_1 \cap e_2 \ |\ e_1, e_2 \in \cE \ \wedge \ e_1 \neq e_2  \ \wedge \ |e_1 \cap e_2 |\geq2\}$.

Let $e = \{ v_1, v_2, \ldots, v_l \} \in \eiE$ be a hyperedge in $EI(\cH)$. By definition, in $\cH$ there exist (at least) two hyperedges $e_1, e_2 \in \cE(\cH)$ both containing all the vertices $v_1, v_2, \ldots, v_l$, more precisely $\{ v_1, v_2, \ldots, v_l \} = e_1 \cap e_2 $ -- we say, $e_1$ and $e_2$ {\em generate} the hyperedge $e  \in \eiE$. In this sense, the hyperedges of $EI(\cH)$ describe sets $\{ v_1, v_2, \ldots, v_l \}$ of vertices having a certain, "strong" neighborhood relation in the original hypergraph $\cH$.

For an application as well as to distinguish the edge intersection hypergraph  from the notion {\em intersection graph} (known from the literature) see \cite{STEIH} and \cite{ST}, respectively.

Obviously, for certain hypergraphs $\cH$ the edge intersection hypergraph $EI(\cH)$  can be 2-uniform; in this case $EI(\cH)$ is a simple, undirected graph $G$ with $V(G) = V(\cH)$.
Note that we consistently use our notion "edge intersection hypergraph" also when the hypergraph $EI(\cH)$ is 2-uniform.

The question, which hypergraphs are edge intersection hypergraphs $EI(\cH)$, seems to be  a difficult one. To simplify the problem, as a first step we can restrict ourselves to the case that the edge intersection hypergraph $EI(\cH)$ is $2$-uniform. Moreover, the situation that the underlying hypergraph $\cH$ is $k$-uniform ($k \ge 3$) may also be easier to handle than the general case.

\medskip
{\bf Problem 1.}
 Find classes of graphs being edge intersection hypergraphs of $k$-uniform $(k \ge 3)$ or non-uniform hypergraphs.

\medskip
In \cite{STEIH} the trees and a large class of cacti, respectively, being edge intersection hypergraphs of 3-uniform hypergraphs, have been characterized.

\medskip
For simplification, we often identify the vertices $v_1, v_2, \ldots, v_n$ of a hypergraph $\cH = (V, \cE)$ with their indices, then we have $V = \{ 1,2, \ldots, n \}$.
In general,  the vertices in $V$ will be always taken modulo $n$.

Let $C_n = ( V, E)$ be the cycle having $n$ vertices, such that $V = \{ 1,2, \ldots, n \}$ and $ \cE = \{ \{ i, i+ 1 \} \, | \, i \in V \}$. For $n \ge 5$, it is trivial to find a 3-uniform hypergraph  $\cH$ with $EI(\cH) = C_n$ (see Theorem \ref{TheoC3U} in Section 2). The analogous  problem becomes much more difficult if we consider $k$-uniform hypergraphs $\cH$ for $k > 3$. In \cite{ST} the 6-uniform case has been investigated in combination with the aim to minimize the number of hyperedges in the hypergraph $\cH$. In a certain sense, the minimization of $|\cE(\cH)|$  corresponds to an as simple as possible structure of the underlying hypergraph $\cH$.

This leads to a general question.

\medskip
{\bf Problem 2.}
Let ${\cal G}$ be a class of graphs, $k \ge 3$, $n_0 \in \N^+ $, $n \ge n_0$ and $G_n \in {\cal G}$ a graph with $n$ vertices. What is the minimum cardinality $| \cE |$ of the edge set of a $k$-uniform hypergraph $\cH_n =(V, \cE)$ with $EI( \cH_n) = G_n$?

\medskip

In the following, we exclusively deal with $k$-uniform hypergraphs $\cH$ having the edge intersection hypergraph $EI(\cH) = C_n$. That is in reference to Problem 2,  ${\cal G}$ is the class of cycles $C_n$. The minimum cardinality of the edge set of a $k$-uniform hypergraph $\cH$ with $EI( \cH) = C_n$ will be denoted by $\mu^k_n$.
 Let us mention that many of the results of the present paper come from \cite{Pae}.

First of all, in Section 2 we  give a sufficient condition guaranteeing the minimum cardinality of the edge set $\cE(\cH)$, where $k \ge 3$,
and formulate a corollary which includes $\mu^k_n  \ge \lceil \frac{3 n}{k} \rceil$.
In combination with Theorem \ref{TheoC3U} (Corollary 4 in \cite{STEIH}) and Theorem \ref{TheoC6U} (Theorem 2 in \cite{ST}), this solves Problem 2 (for the cycle $C_n$) in the 3-uniform and the 6-uniform case, respectively. Thereby, we obtain $\mu^3_n = n$ and  $\mu^6_n  = \lceil \frac{n}{2} \rceil$ for $n \ge 5$ and $n \ge 24$, respectively.

The gap between the 3-uniform and the 6-uniform case will be closed (by  results from \cite{Pae}) in Section 3 and (partially) in Section 4, where we construct  4-uniform and 5-uniform hypergraphs $\cH$ with $EI(\cH)=C_n$, respectively.
In the 4-uniform case, for our constructions we need $n \ge 12$ and we have to consider each of the cases $ n \equiv 0, 1, 2 \mbox{ or } 3 \mbox{ mod } 4$ separately. The proofs of the corresponding results are quite long, since they require very detailed case distinctions -- but we succeed to verify $\mu^4_n  = \lceil \frac{3 n}{4} \rceil$.

In Section 4, the 5-uniformity of $\cH$ further complicates things.
In the first part of the Section, we presuppose $n \ge 18$ and use preferably so-called $(3,2)$-hyperedges $e \in \cE(\cH)$ having a special structure. Whereas -- for arbitrary hyperedges in $\cH$ -- the bound $\lceil \frac{3 n}{5} \rceil$ for $\mu^5_n$ is sharp for infinitely many hypergraphs $\cH$, the exclusive usage of $(3,2)$-hyperedges implicates that the hypergraphs $\cH$ with $EI(\cH) = C_n$ have to have at least $\lceil \frac{2 n}{3} \rceil > \lceil \frac{3 n}{5} \rceil$ hyperedges -- and the bound $\lceil \frac{2 n}{3} \rceil$ is sharp for infinitely many hypergraphs, too. At the end of this Section, we give a special construction for $n \ge 20$ and $ n \equiv 0  \mbox{ mod } 5$ using a larger variety of hyperedges, which allows to reduce the cardinality of the edge set $\cE(\cH)$ to the minimum value $\lceil \frac{3 n}{5} \rceil$. We conjecture $\mu^5_n = \lceil \frac{3 n}{5} \rceil$ for all $n \ge 20$.

The voluminous  proofs of the results of Sections 3 and 4  contain lots of cases to consider and can be found in full length in \cite{Pae}. Therefore, in the present paper  we give only the construction of the hypergraph $\cH$ for each case as well as an (illustrated) example for the construction (with minimum number $n$ of the vertices).

\medskip
At the end of the introduction, let us mention a  tool, which is useful for the investigation of small examples. For this end let $G=(V,E)$ and $\cH =(V, \cE)$ be a graph and a hypergraph, respectively, having one and the same vertex set $V$.
The verification of $\cE(EI(\cH)) = E(G)$ can be done  by hand
 or by computer, e.g. using the computer algebra system MATHEMATICA$^{\textregistered}$
 (\cite{WMath}) with the function
\\[0.8ex]
\indent
$EEI[eh\_] :=
 Complement[
  Select[Union[Flatten[Outer[Intersection, eh, eh, 1], 1]],$
\vspace{-1ex}
\begin{flushright} $   Length[\#] > 1 \&], eh],$ \end{flushright}

%\vspace{-2ex}
\noindent
where the argument $eh$ has to be the list of the hyperedges of $\cH$ in the form $\{\{a,b, \ldots, c \}, \ldots,$ $\{x,y, \ldots, z \} \}$. Then $EEI[eh]$ provides the list of the hyperedges of $EI(\cH)$.

%\pagebreak

\section{Some preliminary results}

At first, we give a -- simple, but very useful -- sufficient condition for $| \cE(\cH) |$ to be minimum.

\begin{theorem}[\cite{Pae}]
\label{TheoOPT3REG}
Let $\mathcal{H}=(V,\mathcal{E})$ be $k$-uniform with $EI(\mathcal{H})=C_n$.
If $\mathcal{H}$ is $3$-regular, then $\mathcal{H}$ has a minimum number  $|\mathcal{E}|$ of hyperedges for all  $k$-uniform hypergraphs $\mathcal{H}'$ with $EI(\mathcal{H}')=C_n$, i.e. $|\mathcal{E}| = \mu^k_n$.
\end{theorem}

\begin{proof}
Since $\mathcal{H}$ is $3$-regular, every vertex $v$ is contained in exactly three hyperedges of $\mathcal{H}$. In connection with the $k$-uniformity we obtain  $3 n = \sum_{v\in V} d_{\mathcal{H}}(v)=|\mathcal{E}|\cdot k$.\\
Assume, there is a $k$-uniform hypergraph $\mathcal{H}'=(V,\mathcal{E}')$ with $EI(\mathcal{H}')=C_n$ and $|\mathcal{E}'|<|\mathcal{E}|$. Then $|\mathcal{E}'|\cdot k <|\mathcal{E}|\cdot k =3n$ implies the existence of a vertex $u$ in  $\mathcal{H}'$ having the degree $d_{\mathcal{H}'}(u)<3$. This contradicts  $EI(\mathcal{H}')=C_n$:
 In $C_n$ every vertex $u$ is incident to two distinct edges $e, e' \in E(C_n) = \cE(EI(\mathcal{H}'))$;
  therefore in $\mathcal{H}'$ there have to be hyperedges $e_1, e_2, e'_1, e'_2 \in \cE(\mathcal{H}')$ such that
   $e_1 \cap e_2 = e$ and $e'_1 \cap e'_2 = e'$, where $| \{ e_1, e_2, e'_1, e'_2 \} | \ge 3$. Clearly, $u$ has to be contained in each of the hyperedges $e_1, e_2, e'_1$ and $ e'_2$, i.e.  $d_{\mathcal{H}'}(u) \ge 3$.
   \hqed
\end{proof}

Looking at the end of the proof, for the situation that two hyperedges $e_1, e_2 \in \cE$ generate the edge $e = e_1 \cap e_2 \in E(C_n)$, we also say that $e_1$ and  $e_2$ {\em contribute} to $e$ -- they serve as {\em half-edges} in  $EI(\mathcal{H})$.

The proof of Theorem \ref{TheoOPT3REG} implies an interesting lower bound for the minimum number of hyperedges in the $k$-uniform case.

\begin{corollary}[\cite{Pae}]
\label{CorMIN_l_uniform}
Let $\mathcal{H}=(V,\mathcal{E})$ be $k$-uniform with $EI(\mathcal{H})=C_n$. Then $| \cE | \ge \frac{3 n}{k}$, i.e. $\mu^k_n  \ge \frac{3 n}{k}$.
\end{corollary}

\begin{proof}
From the above proof we obtain  $d_{\mathcal{H}}(v) \ge 3$ for each $ v \in \cE$. Consequently,
$|\mathcal{E}|\cdot k = \sum_{v\in V} d_{\mathcal{H}}(v) \ge 3 n$ and  $| \cE | \ge \frac{3 n}{k}$. \hqed
\end{proof}

The following (easy to prove) result for 3-uniform hypergraphs can be found as Corollary 4 in \cite{STEIH}.

\begin{theorem}[\cite{STEIH}]
\label{TheoC3U}
Let $n \ge 5$. Then there exists a 3-regular and 3-uniform hypergraph $\cH=(V, \cE)$  with  $EI(\cH) = C_n$ and $|\cE| = n$.
Consequently, $\mu^3_n = n$.
\end{theorem}

{\em To the proof.} The hypergraph  $\cH=(V, \cE)$ with $ \cE = \{ \{i, i+1, i+2 \} \, | \, i \in \{ 1, 2, \ldots, n \} \}$ has the required properties.
\hqed

\bigskip
Firstly, in the 6-uniform case, we cite the (not easy to prove!) Theorem 2 from \cite{ST}.

\begin{theorem}[\cite{ST}]
\label{TheoC6U}
Let $n \ge 24$. Then there exists a hypergraph $\cH=(V, \cE)$  with  $EI(\cH) = C_n$ such that the following holds.
\begin{enumerate}
\item[(i)]
If $n$ is even, then $\cH$ is 3-regular, 6-uniform and $|\cE| = \frac{n}{2}$.
\item[(ii)]
If $n$ is odd, then $\cH$ is 3-regular, $|\cE| = \frac{n+1}{2}$, $\cH$ contains one hyperedge
%${\underline e}$
of cardinality 3 and all other hyperedges in $\cH$ have cardinality 6.
\end{enumerate}
\end{theorem}

In the case $n$ even, note that the number of hyperedges $|\cE| = \frac{n}{2}$ of the hypergraph $\cH$ is minimum for all %3-regul\"{a}r and
6-uniform hypergraphs $\cH'=(V, \cE')$  with  $EI(\cH') = C_n$. This follows from Corollary \ref{CorMIN_l_uniform}. Besides, in \cite{ST} this had been proved in a different way, namely by discussing all possible types of hyperedges which can occur in $\cH$.
It is easy to see that also in the case $n$ odd $\cH$  has a minimum number of hyperedges. The reason can be understood by having a look at the construction of the cardinality-3-hyperedge in the proof of Lemma 3 in \cite{ST}:

First, using the construction described in the proof of Theorem 2 in \cite{ST}, we construct a hypergraph $\cH' = ( V', \cE')$ with $V' = \{ 1,2, \ldots, n-1 \}$, $\cE' = \{e_1', e_2', \ldots, e_{\frac{n-1}{2}}' \} $ having the edge intersection hypergraph $C_{n-1}$. By Theorem \ref{TheoC6U}(i), $\cH'$ has a minimum  number of hyperedges, namely $\frac{n-1}{2}$.

Secondly, to get the hypergraph $\cH$ with $EI(\cH) = C_n$, we add
a new vertex $n$ "between" the vertices 3 and 4 in this cycle, i.e. in $C_{n-1}=EI(\cH')$, and obtain $\cH=(V, \cE)$ by the following  construction.

\medskip
$ V:= V' \, \cup \, \{ n \}$,

\medskip
$ \cE := \{ e_1, e_2, \ldots, e_{\frac{n+1}{2}} \}$, where

\bigskip
%\hspace*{5cm}
$e_i = \left\{ \begin{array}{l@{\;,\quad}l}
 e_i' & i = 1,4,5,6, \ldots, \frac{n-1}{2} \\[0.8ex]
(e_2' \, \setminus \, \{ 4 \}) \, \cup \, \{ n \} & i = 2 \\[0.8ex]
(e_3' \, \setminus \, \{ 3 \}) \, \cup \, \{ n \} & i = 3  \\[0.8ex]
\{ 3, n, 4 \} & i = \frac{n+1}{2} \, .
 \end{array} \right.$

\bigskip

Since in $\cH$ (as well as in $EI(\cH) = C_n$) we have one vertex more than in $\cH'$ (as well as in $C_{n-1}=EI(\cH')$), in $C_n$ we need one edge more than in $C_{n-1}$. For this end, in $\cH$ we need two half-edges more than we have in $\cH'$. For this end, at least one additional hyperedge is needed in $\cH'$ in comparison with $\cH$ -- this is the (new) cardinality-3-hyperedge $e_{\frac{n+1}{2}}$. Note that $e_2, e_3 \in \cE$ are only slight modifications of $e_2', e_3' \in \cE'$ (for details see \cite{ST}).

\medskip

Giving up the 3-regularity we can enforce the 6-uniformity of $\cH$.

\begin{corollary}
\label{CorC6U}
Let $n \ge 24$. Then there exists a 6-uniform hypergraph $\cH=(V, \cE)$  with  $EI(\cH) = C_n$ and $|\cE| = \lceil \frac{n}{2} \rceil $. Moreover, $\mu^6_n = \lceil \frac{n}{2} \rceil$.
\end{corollary}

{\em To the proof.}  For $n$ even there is nothing to do. \\
For $n$ odd, it suffices to add three vertices $u, v, w \in V$ to the cardinality-3-hyperedge $e_{\frac{n+1}{2}}$ in order to obtain a new hyperedge $\widetilde{e}_{\frac{n+1}{2}}$ having cardinality 6, such that in the resulting edge intersection hypergraph no additional edges are generated by $\widetilde{e}_{\frac{n+1}{2}}$.
Since  $n \ge 25$ holds,  by means of a  case distinction  the existence of suitable vertices $u, v, w \in V$ can be shown -- for shortness, we omit the further details of this proof.

Note that adding the vertices $u, v $ and $w $ to the hyperedge $e_{\frac{n+1}{2}}$ destroys the 3-regularity of the hypergraph, whereas $|\cE| = \lceil \frac{n}{2} \rceil = \frac{n+1}{2}$, i.e. the minimality of $|\cE|$, remains valid.
\hqed

\medskip

\section{The 4-uniform case}

In the Sections 3 and 4, we will make use of the following two notations.

For $i \in \{ 1,2, \ldots, n \}$ and $e \in \cE$, a sequence $(i, i+1, \ldots, i+k-1)$ with $\{i, i+1, \ldots, i+k-1 \} \subseteq e$, such that $i-1 \notin e$ and $i+k \notin e$, is referred to as a {\em $k$-section } of $e$ on $C_n$.

For $t \ge 1$, a hyperedge $e \in \cE$ is an {\em $(l_1, l_2, \ldots, l_t)$-hyperedge}, if and only if $e$ consists of $t$ $l_i$-sections of the cardinalities $ l_1 \ge l_2 \ge \ldots \ge l_t$.

%\smallskip
%Let $\cH=(V, \cE)$  be 4-uniform.

\begin{theorem}[\cite{Pae}]
\label{TheoMinVertex4U}
If $\cH = ( V, \cE)$ is 4-uniform with  $EI(\cH) = C_n$, then $n \ge 11$.
\end{theorem}

{\em To the proof.}  The verification of $n > 7$ is simple. To show $n \neq 8$, $n \neq 9$ and $n \neq 10$ detailed case distinctions are necessary, where all possible combinations of $k$-sections of the hyperedges $e \in \cE$ of the  hypergraph $\cH$, which are needed to generate the edges of $C_n = EI(\cH)$, have to be discussed. So for $n= 10$, thirteen subcases have to be considered. \hqed

\medskip
For $n \ge 11$ a suitable hypergraph $\cH$ generating $EI(\cH) = C_n$ is given in the sketch of the proof of the following Theorem.

\begin{theorem}[\cite{Pae}]
\label{Theo11_4U}
Let $n \ge 11$. Then there exists a 4-uniform hypergraph $\cH = ( V, \cE)$  with  $EI(\cH) = C_n$ and $| \cE | = n$.
\end{theorem}

{\em To the proof.} Choose $\cE = \{ \{ i, i+1, i + 2, i + 5 \} \, | \, i \in \{ 1, 2, \ldots, n \} \}$. \hqed
\\[0.5ex]

{\em Example.} For $n = 11$ we have $\cH = ( V, \cE)$ with $V = \{ 1, 2, \ldots, 11 \}$ and \\[1ex]
$\cE = \{ \{1,2,3,6\},\{2,3,4,7\},\{3,4,5,8\},\{4,5,6,9\},\{5,6,7,10\},\{6,7,8,11\},\{7,8,9,1\},\{8,9,10,2\},$ \\
\phantom{$\cE = \{ $}$\{9,10,11,3\},
\{10,11,1,4\},\{11,1,2,5\}\}$.

\medskip
In principle, we use the same construction as in the proof of Theorem \ref{TheoC3U} (Corollary 4 in \cite{STEIH}). We only have to add a fourth, "innocuous" vertex to the cardinality-3-hyperedges from \cite{STEIH} in order to guarantee the 4-uniformity of our present hypergraph (cf. Fig. 1). This is the reason for the minimum cardinality 11 of the vertex set we need here.
\begin{figure}[h]
	\centering
	\includegraphics[width=11cm]{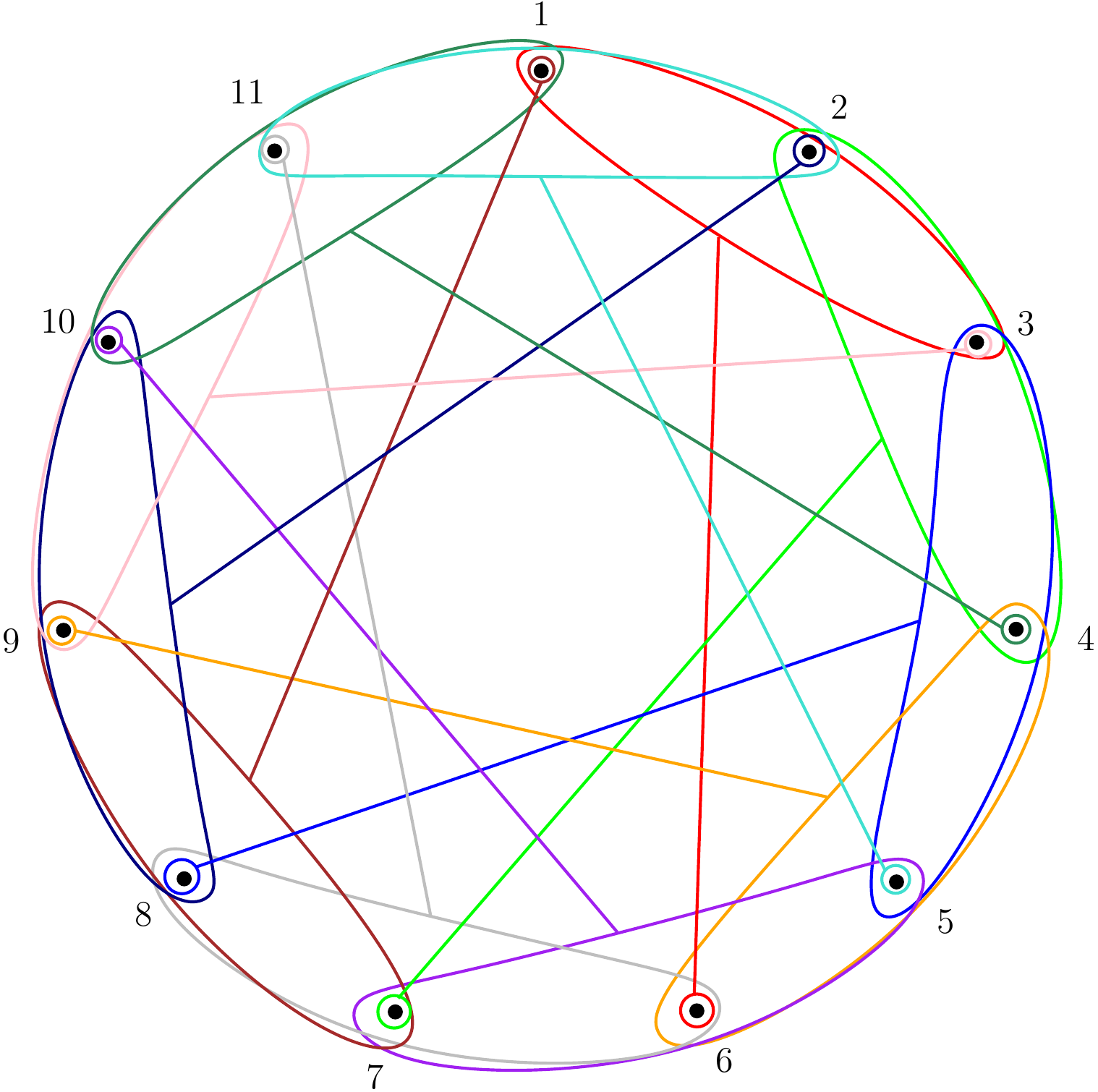}
	\caption{A $4$-uniform hypergraph $\cH=(V,\cE)$ mit $EI(\mathcal{H})=C_{11}.$}
    \label{img:Bsp11Knoten}
\end{figure}

\bigskip
%\vspace{1cm}
To obtain results with  minimum cardinality of the set $\cE(\cH)$ of hyperedges, we assume $n \ge 12$ and consider each of the cases
$n \equiv i \mbox{ mod } 4$ separately, for all $i \in \{ 0,1,2,3 \}$.

\begin{theorem}[\cite{Pae}]
\label{TheoMain_4U}
Let $n \ge 12$. Then there exists a 4-uniform hypergraph $\cH=(V, \cE)$  with  $EI(\cH) = C_n$ and $|\cE| = \lceil \frac{3}{4} n \rceil $. Therefore, $\mu^4_n = \lceil \frac{3}{4} n \rceil$.
\end{theorem}

{\em To the proof.}
As mentioned above,  in each case we give only the construction of a suitable hypergraph $\cH = ( V, \cE)$ (for every possible $n \ge 12$) as well an example with the minimum number $n$ of the vertices. The detailed proof of $\cE(EI(\cH))  \subseteq E(C_n)$ always causes much more effort than the verification of $E(C_n) \subseteq \cE(EI(\cH))$ (see \cite{Pae}). In other words, the more difficult part of the proof is to show that the hyperedges of $\cH$ do not generate "unwanted" (hyper-)edges in $EI(\cH)$.

%Before coming to the construction of $\cH$, we mention that to show $\mu^4_n = \lceil \frac{3}{4} n \rceil$ requires to discuss, how many edges in $EI(\cH)$ can be generated by the different types of hyperedges in $\cH$ in each of the cases under consideration (see \cite{Pae}).

%Now -- for all four cases -- we give the edge set $\cE$ of the desired hypergraph $\cH = ( V, \cE)$ and an example with minimum $n$.

Now we discuss the four possible cases for $n = |V|$.
%\pagebreak

\medskip
\noindent
{\ul{\em Case 0: $n \equiv 0 \mod 4$.}}

\smallskip

In the present case, to construct the  hypergraph $\cH = ( V, \cE)$, we use two types of hyperedges. The first type consists of a 4-section and the second one contains two 2-sections.

\noindent
$\mathcal{E}=\{e_i=\{i,i+1,i+2,i+3\} \text{ }| \text{ }i \in \{1,3,\ldots,n-1\}\}  \quad \cup $ \\[0.5ex]
\phantom{$\mathcal{E}=$ }$ \{e_i=\{i,i+1,\frac{n}{2}+i,\frac{n}{2}+i+1\} \text{ }| \text{ }i \in \{2,4,\ldots,	
\frac{n}{2}\}\}$.

\bigskip
{\em Example.} For $n = 12$ we have $\cH = ( V, \cE)$ with $V = \{ 1, 2, \ldots, 12 \}$ and \\[1ex]
$\mathcal{E}=\{ \{1,2,3,4\},\{3,4,5,6\},\{5,6,7,8\}, \{7,8,9,10\},\{9,10,11,12\},\{11,12,1,2\}$ \\[0.5ex]
\phantom{$\mathcal{E}= \{$}$ \{2,3,8,9\}, \{4,5,10,11\}, \{6,7,12,1\} \}$.
\begin{figure}[h]
	\centering
	\includegraphics[width=10cm]{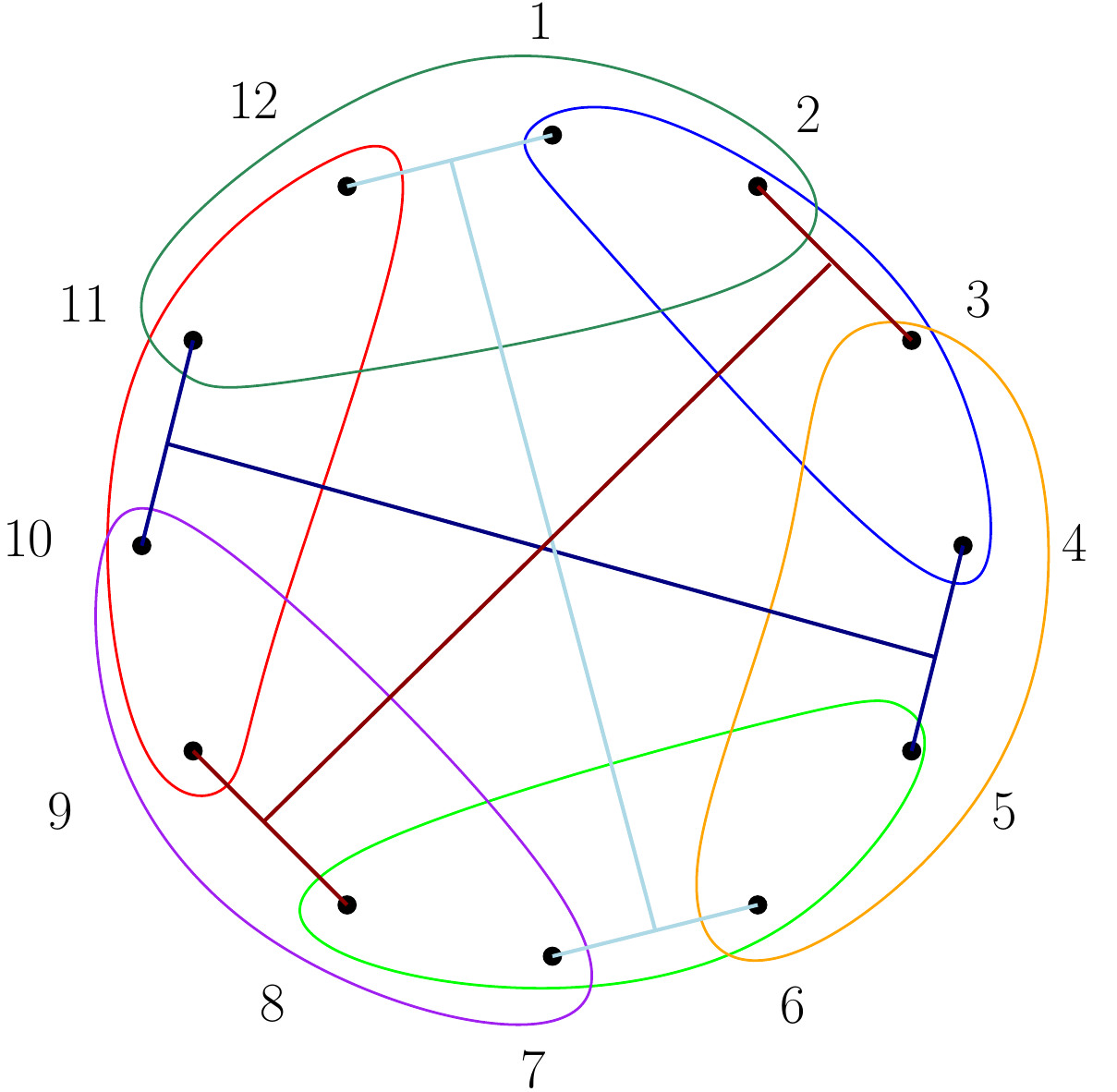}
	\caption{A $4$-uniform hypergraph $\mathcal{H}=(V,\mathcal{E})$ with $EI(\mathcal{H})=C_{12}.$}
	\label{img:Bsp12Knoten}
\end{figure}

%\smallskip
%\vspace{1cm}
\noindent
{\ul{\em Case 1: $n \equiv 1 \mod 4$.}}

\smallskip

Whereas in Case 0 we had two different types of hyperedges in $\cH$, namely $(4)$-hyperedges and $(2,2)$-hyperedges, in the present Case 1 we need additionally a $(3,1)$-hyperedge as a third type.

\medskip
\noindent
$\mathcal{E}=\{e_i=\{i,i+1,i+2,i+3\} \text{ }| \text{ }i \in \{1,3,\ldots,n-2 \}\} \quad \cup $ \\[0.5ex]
\phantom{$\mathcal{E}=$ }$ \{e_i=\{i,i+1,\frac{n-1}{2}+i,\frac{n-1}{2}+i+1\} \text{ }| \text{ }i \in \{2,4,\ldots,	
	\frac{n-1}{2}\}\} \quad \cup $ \\[0.5ex]
\phantom{$\mathcal{E}=$ }$ \{e_n=\{n,1,2,5\}\}$.

\bigskip
{\em Example.} For $n = 13$ we have $\cH = ( V, \cE)$ with $V = \{ 1, 2, \ldots, 13 \}$ and \\[1ex]
$\mathcal{E}=\{ \{1,2,3,4\},\{3,4,5,6\},\{5,6,7,8\},\{7,8,9,10\},\{9,10,11,12\},\{11,12,13,1\},$ \\[0.5ex]
\phantom{$\mathcal{E}= \{$}$ \{2,3,8,9\},\{4,5,10,11\},\{6,7,12,13\},$ \\[0.5ex]
\phantom{$\mathcal{E}= \{$}$ \{13,1,2,5\}\}$.
%\pagebreak
% $\cH = ( V, \cE)$ is shown in Fig. 3.

%\vspace{-12mm}
\begin{figure}[h]
	\centering
	\includegraphics[width=11cm]{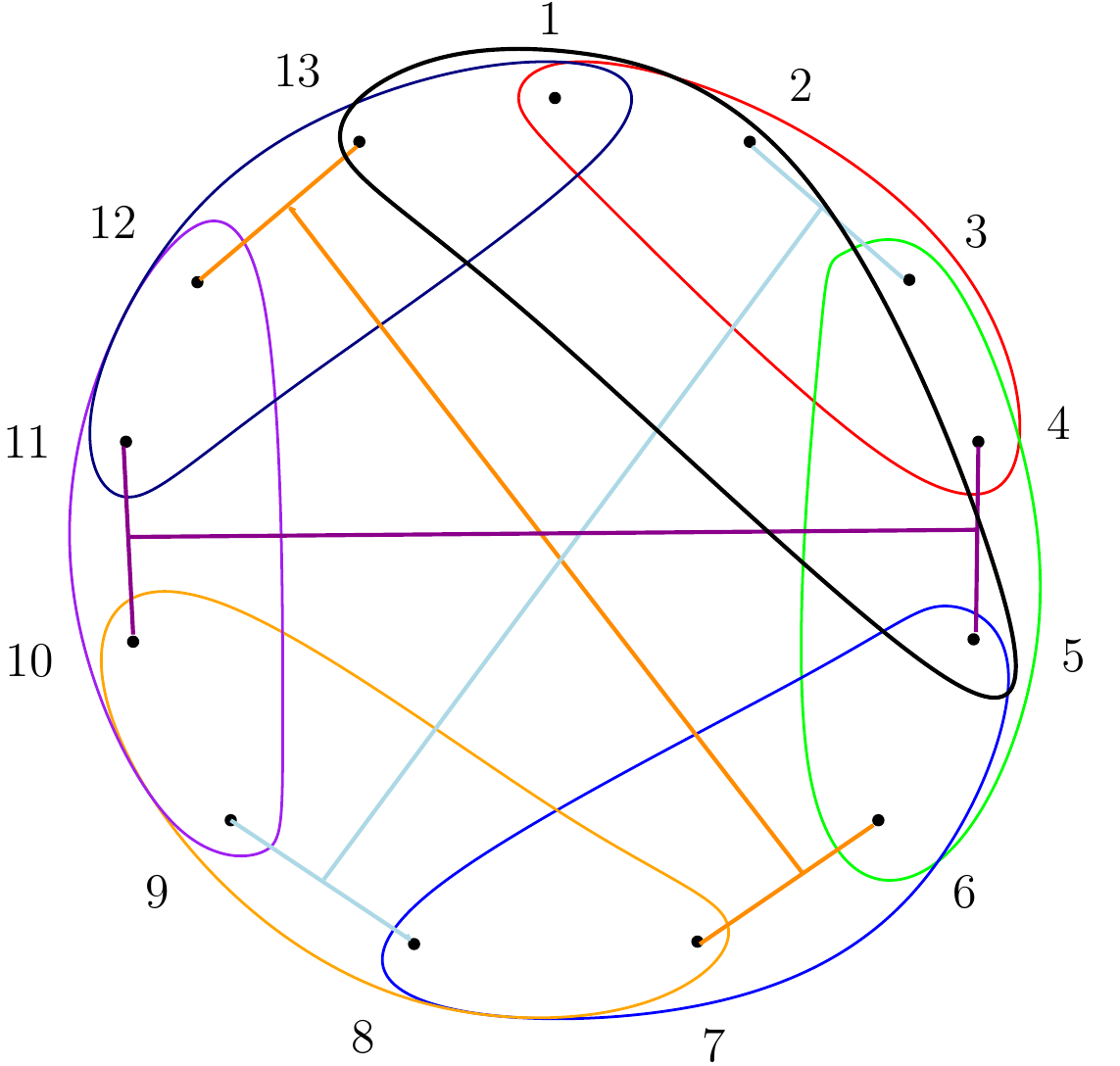}
	\caption{A $4$-uniform hypergraph $\mathcal{H}=(V,\mathcal{E})$ with $EI(\mathcal{H})=C_{13}.$}
	\label{img:Bsp13Knoten}
\end{figure}
%\vspace{1cm}
%\pagebreak
\noindent
{\ul{\em Case 2: $n \equiv 2 \mod 4$.}}

\smallskip

As in Case 0 (and in contrast to Case 1), in our construction only $(4)$-hyperedges and $(2,2)$-hyperedges will be used.

\smallskip

\noindent
$\mathcal{E}=\{e_i=\{i,i+1,i+2,i+3\} \text{ }| \text{ }i \in \{1,3,...,n-1 \}\}\quad \cup $ \\[0.5ex]
\phantom{$\mathcal{E}=$ }$ \{e_i=\{i,i+1,\frac{n}{2}+i-1,\frac{n}{2}+i\} \text{ }| \text{ }i \in \{2,4,...,	
	\frac{n}{2}-1\}\} \quad \cup $ \\[0.5ex]
\phantom{$\mathcal{E}=$ }$ \{e_n=\{n,1,\frac{n}{2}+1,\frac{n}{2}+2\}\}$.

\smallskip

 But now two of the $(2,2)$-hyperedges, namely $e_i=\{i,i+1,\frac{n}{2}+i-1,\frac{n}{2}+i\}$ (for $i = 2$), $ e_n=\{n,1,\frac{n}{2}+1,\frac{n}{2}+2\}$ as well as the $(4)$-hyperedge $e_i=\{i,i+1,i+2,i+3\}$ (for $i= \frac{n}{2}$) contain the vertices $\frac{n}{2}+1$ and  $\frac{n}{2}+2$.
 In other words, the hyperedge  $\{\frac{n}{2}+1,\frac{n}{2}+2\}$ is "triply-generated" in the edge intersection hypergraph $EI(\mathcal{H})$ -- that way we take account of the 4-uniformity of the hypergraph $\mathcal{H}$.

\smallskip
%\pagebreak
{\em Example.} For $n = 14$ we have $\cH = ( V, \cE)$ with $V = \{ 1, 2, \ldots, 14 \}$ and \\[1ex]
$\mathcal{E}=\{\{1,2,3,4\},\{3,4,5,6\},\{5,6,7,8\},\{7,8,9,10\},\{9,10,11,12\},\{11,12,13,14\}, \{13,14,1,2\},$ \\[0.5ex]
\phantom{$\mathcal{E}= \{$}$ \{2,3,8,9\},\{4,5,10,11\},\{6,7,12,13\},$ \\[0.5ex]
\phantom{$\mathcal{E}= \{$}$ \{14,1,8,9\}\}.$

\pagebreak
$\cH = ( V, \cE)$ can be seen in Fig. 4; the hyperedge being "triply-generated" is the hyperedge $\{ 8, 9 \} = \{7,8,9,10\} \, \cap \, \{2,3,8,9\} \, = \, \{7,8,9,10\} \, \cap \, \{14,1,8,9\} \, = \,  \{2,3,8,9\} \, \cap \, \{14,1,8,9\}$.
\begin{figure}[h]
	\centering
	\includegraphics[width=11cm]{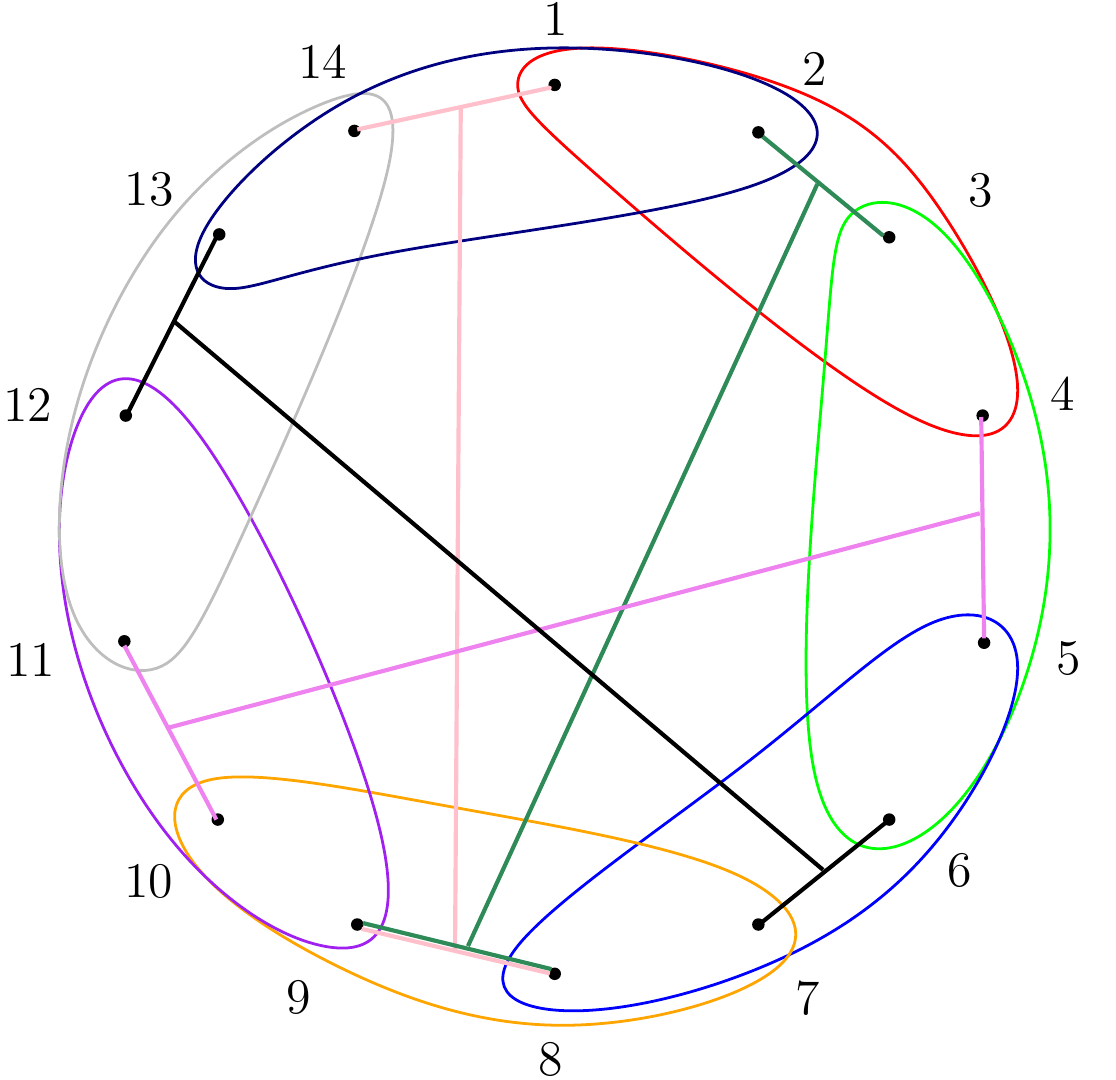}
	\caption{A $4$-uniform hypergraph $\mathcal{H}=(V,\mathcal{E})$ with $EI(\mathcal{H})=C_{14}.$}
	\label{img:Bsp14Knoten}
\end{figure}

%\pagebreak
\noindent
{\ul{\em Case 3: $n \equiv 3 \mod 4$.}}

\smallskip
Also in the last case, we make use of three types of hyperedges, but now the third class of hyperedges consists of two hyperedges instead of one (as in Case 2). This results in some more subcases having to be investigated now (see \cite{Pae}).
At first, here is the set $\cE$ of hyperedges of the hypergraph $\cH$.

\medskip
\noindent
$\mathcal{E}= \{e_i=\{i,i+1,i+2,i+3\} \text{ }| \text{ }i \in \{1,3,...,n-2 \}\}\quad \cup $ \\[0.5ex]
\phantom{$\mathcal{E}=$ }$ \{e_i=\{i,i+1,\frac{n-1}{2}+i-1,\frac{n-1}{2}+i\} \text{ }| \text{ }i \in \{2,4,...,	
	\frac{n-1}{2}-1\}\} \quad \cup $ \\[0.5ex]
\phantom{$\mathcal{E}=$ }$ \{e_{n-1}=\{n-1,n,\frac{n+1}{2},\frac{n+1}{2}+1\},  e_n=\{n,1,2,5\}\}$.

\bigskip
This set of hyperedges $\cE$ has two typical features we have seen above in Case 1 and Case 2, respectively (see also our example below).
\begin{itemize}
\item The first one is the existence of a $(3,1)$-hyperedge ($\{15,1,2,5\}$ in the example).
\item The second one is a "triply-generated" hyperedge in the resulting edge intersection hypergraph ($\{ 8, 9 \} = \{7,8,9,10\} \, \cap \, \{2,3,8,9\} \, = \, \{7,8,9,10\} \, \cap \, \{14,15,8,9\} \, = \,  \{2,3,8,9\} \, \cap \, \{14,15,8,9\}$ in the example). \hqed
    \end{itemize}

\medskip
{\em Example.} For $n = 15$ we have $\cH = ( V, \cE)$ with $V = \{ 1, 2, \ldots, 15 \}$ and \\[1ex]
$\mathcal{E}=\{ \{1,2,3,4\},\{3,4,5,6\},\{5,6,7,8\},\{7,8,9,10\},\{9,10,11,12\},\{11,12,13,14\}, \{13,14,15,1\},$ \\[0.5ex]
\phantom{$\mathcal{E}= \{$}$ \{2,3,8,9\},\{4,5,10,11\},\{6,7,12,13\},$ \\[0.5ex]
\phantom{$\mathcal{E}= \{$}$ \{14,15,8,9\},\{15,1,2,5\}\}.$ \hfill (see Fig. 5)
%\pagebreak
\begin{figure}[h]
	\centering
	\includegraphics[width=11cm]{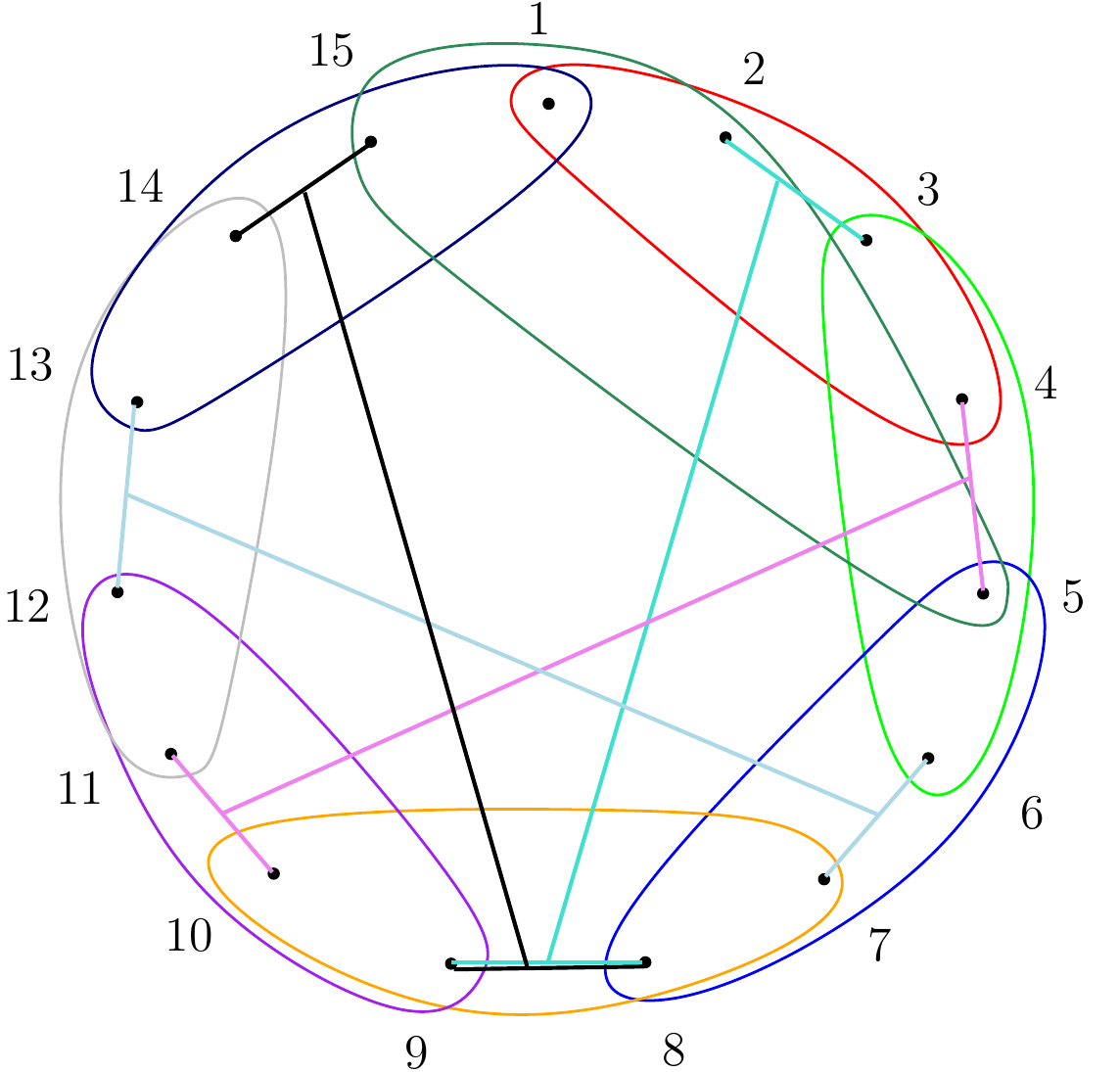}
	\caption{A $4$-uniform hypergraph $\mathcal{H}=(V,\mathcal{E})$ with $EI(\mathcal{H})=C_{15}.$}
	\label{img:Bsp15Knoten}
\end{figure}

\section{The 5-uniform case}

\subsection{$(3,2)$-hyperedges and $| \cE | = \lceil \frac{2 n}{3} \rceil$}

In the 5-uniform case, $(3,2)$-hyperedges are "good" ones in the following sense. They can contribute to generate three edges in $EI(\cH) = C_n$ (they provide three half-edges) and they are "easier to handle" than hyperedges containing a 4-section $(i, i+1, i+2, i+3)$ or a 5-section $(i, i+1, i+2, i+3, i+4)$. To see this, let us have a look at the "middle vertices" $i+1, i+2$ of a 4-section or a 5-section. In order to generate the edge $\{ i+1, i+2 \}$ in $EI(\cH) = C_n$, it is compelling to have a hyperedge with the 2-section $\{ i+1, i+2 \}$ in $\cH$ -- that can be an annoying restriction. So, in a first step, we deal with $(3,2)$-hyperedges.

\begin{theorem}[\cite{Pae}]
\label{TheoMIN_(3,2)}
Let $\mathcal{H}=(V,\mathcal{E})$ be $5$-uniform with $EI(\mathcal{H})=C_n$, such that all hyperedges in $\cH$ are $(3,2)$-hyperedges. Then $| \cE | \ge \frac{2 n}{3}$.
\end{theorem}

\begin{proof}
Considering an arbitrarily chosen $(3,2)$-hyperedge $e$, in $EI(\cH) = C_n$ this hyperedge can contribute to at most 3 edges of $C_n$. Every edge in $EI(\cH)$ is the intersection of at least 2 of the  $(3,2)$-hyperedges. Therefore, we obtain
$|\mathcal{E}|\geq\lceil\frac{2 n}{3}\rceil\geq\frac{2 n}{3}$.   \hqed
\end{proof}

For $n \ge 18$ and $n \equiv 0 \mod 3$ it can be proved constructively that the given lower bound for $| \cE |$ is sharp.

For $n \equiv 1 \mod 3$ and $n \equiv 2 \mod 3$, respectively, we modify this construction in order to obtain the existence of a $5$-uniform hypergraph $\mathcal{H}=(V,\mathcal{E})$ with $EI(\mathcal{H})=C_n$ and $| \cE | = \lceil \frac{2 n}{3} \rceil$, too. In both cases of the modification, additional to the two main types of $(3,2)$-hyperedges this leads to nine extra hyperedges which are needed.  %For $n \equiv 1 \mod 3$ and $n \equiv 2 \mod 3$,
Besides (3,2)-hyperedges, one of the extra hyperedges is a (5)-hyperedge. Moreover, for $n \equiv 1 \mod 3$ we need another extra hyperedge being a (2,2,1)-hyperedge.

At first, we consider each of the three cases in a separate lemma.
Let us mention that, as in the previous Section, the detailed verification of $\cE(EI(\cH))= E(C_n)$ would transcend the limitations of the present paper.
%, see our remarks to the proofs of Lemma \ref{Lemma_(3,2)_1_mod_3} and Lemma \ref{Lemma_(3,2)_2_mod_3}.
Hence, for shortness we construct only $\cE$ as the basic step to show the existence of the wanted hypergraphs $\mathcal{H}=(V,\mathcal{E})$ with $EI(\mathcal{H})=C_n$. The proof in all details can be found in \cite{Pae}.

Moreover, for each of our constructions, an example with minimum number $n$ of vertices will be given.

\begin{lemma}[\cite{Pae}]
\label{Lemma_(3,2)_0_mod_3}
Let $n \ge 18$ and $n \equiv 0 \mod 3$. Then there exists a $5$-uniform hypergraph $\mathcal{H}=(V,\mathcal{E})$ with $EI(\mathcal{H})=C_n$ and $| \cE | = \frac{2 n}{3}$ such that all hyperedges in $\cH$ are $(3,2)$-hyperedges.
\end{lemma}

{\em To the proof.} We have to consider  $n = 18$ separately.

\smallskip
\noindent
{\ul{\em Case 1:  $n = 18$.}}

\smallskip

Let
$\mathcal{H}=(V, \mathcal{E})$ with
$V=\{1,2,...,18\}$ and \\[1ex]
$\mathcal{E}=\{ \{1,2,3,7,8\},\{4,5,6,12,13\},\{7,8,9,13,14\},
\{10,11,12,18,1\},\{13,14,15,1,2\}, $\\[0.5ex]
\phantom{$\mathcal{E}= \{$}$ \{16,17,18,6,7\},$\\[0.5ex]
\phantom{$\mathcal{E}= \{$}$ \{2,3,4,9,10\},\{5,6,7,10,11\},\{8,9,10,15,16\},\{11,12,13,16,17\},\{14,15,16,3,4\}, $\\[0.5ex]
\phantom{$\mathcal{E}= \{$}$ \{17,18,1,4,5\}\}.$

\medskip
\noindent
The hypergraph is drawn in Fig. 6.
\vspace{-2mm}
\begin{figure}[h]
	\centering
	\includegraphics[width=11.5cm]{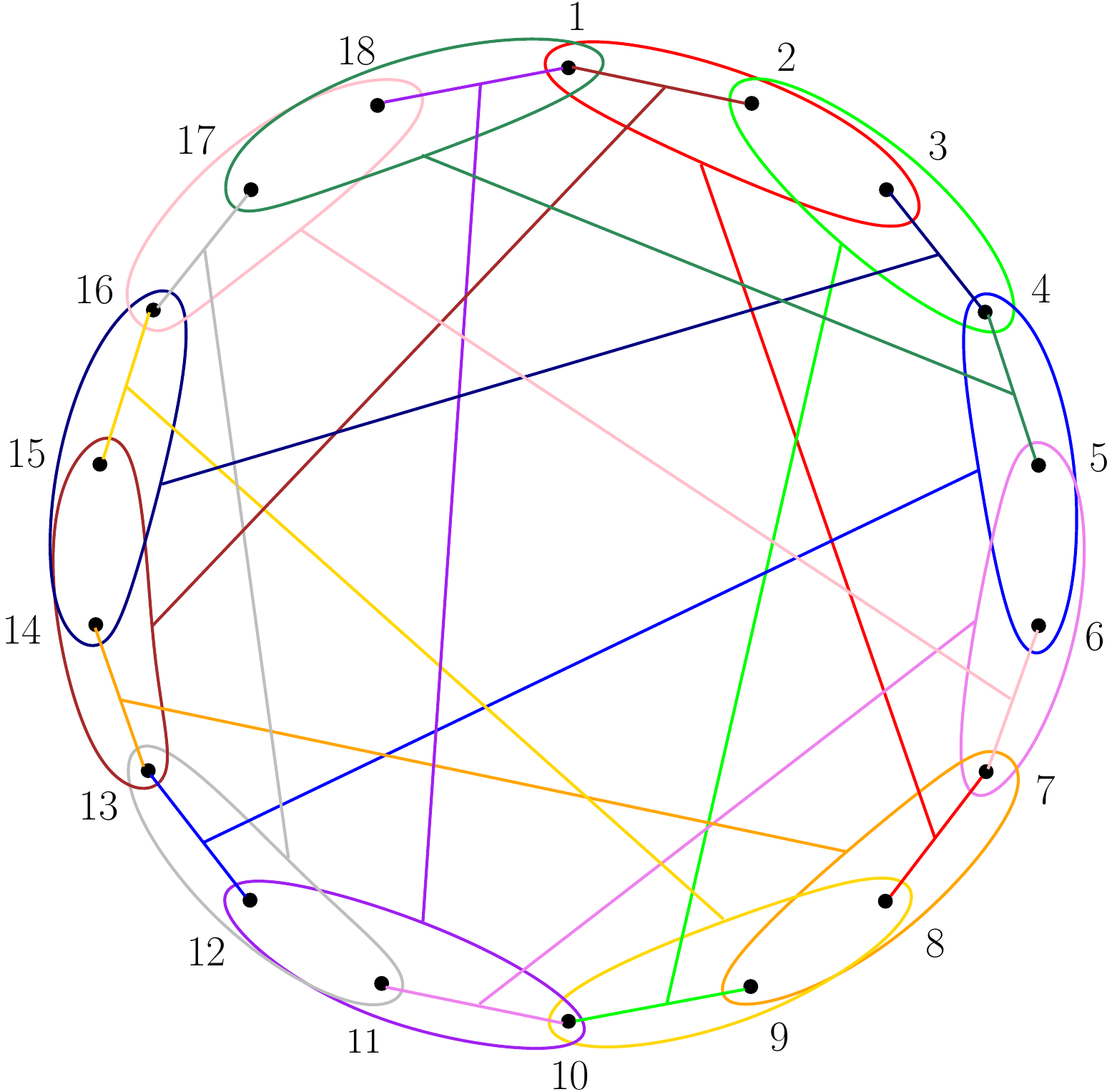}
	\caption{A $5$-uniform hypergraph $\mathcal{H}=(V,\mathcal{E})$ with $EI(\mathcal{H})=C_{18}$.}
	\label{img:Bsp18Knoten}
\end{figure}

\pagebreak
\noindent
{\ul{\em Case 2:  $n \ge 21$.}}

\smallskip

Let
$\mathcal{H}=(V, \mathcal{E})$ with
$V=\{1,2,...,n\}$ and \\[1ex]
$\mathcal{E}=\{ e_i=\{i,i+1,i+2,i+8,i+9\} \text{ }| \text{ }i \in \{1,4,7,...,n-2\}\}\quad \cup $ \\[0.5ex]
\phantom{$\mathcal{E}=$ }$ \{e_i=\{i,i+1,i+2,i+5,i+6\} \text{ }| \text{ }i \in \{2,5,8,...,	
n-1\}\}$. \hqed

\medskip
{\em Example.} For $n = 21$ we have $\cH = ( V, \cE)$ with $V = \{ 1, 2, \ldots, 21 \}$ and \\[1ex]
$\mathcal{E}=\{ \{1,2,3,9,10\},\{4,5,6,12,13\},\{7,8,9,15,16\},\{10,11,12,18,19\},\{13,14,15,21,1\},$\\[0.5ex]
\phantom{$\mathcal{E}= \{$}$ \{16,17,18,3,4\},\{19,20,21,6,7\},$\\[0.5ex]
\phantom{$\mathcal{E}= \{$}$ \{2,3,4,7,8\},\{5,6,7,10,11\},\{8,9,10,13,14\},\{11,12,13,16,17\}, \{14,15,16,19,20\},$\\[0.5ex]
\phantom{$\mathcal{E}= \{$}$ \{17,18,19,1,2\},\{20,21,1,4,5\}\}.$

\medskip
\noindent
Figure 7 shows $\cH = ( V, \cE)$.
\vspace{-8mm}
\begin{figure}[h]
	\centering
	\includegraphics[width=11.5cm]{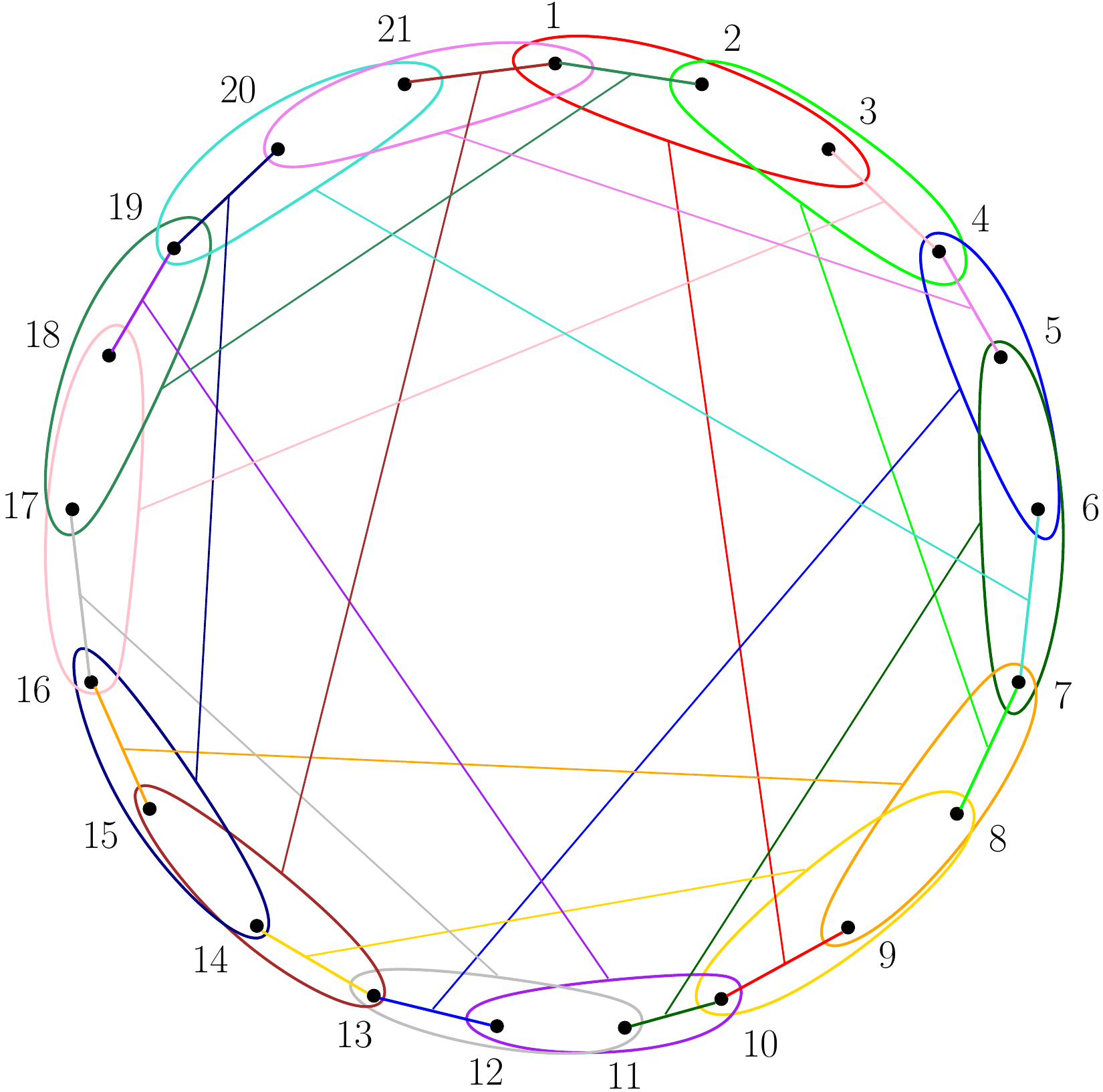}
	\caption{A $5$-uniform hypergraph $\mathcal{H}=(V,\mathcal{E})$ with $EI(\mathcal{H})=C_{21}$.}
	\label{img:Bsp21Knoten}
\end{figure}
\vspace{-5mm}
\begin{lemma}[\cite{Pae}]
\label{Lemma_(3,2)_1_mod_3}
Let $n \ge 19$ and $n \equiv 1 \mod 3$. Then there exists a $5$-uniform hypergraph $\mathcal{H}=(V,\mathcal{E})$ with $EI(\mathcal{H})=C_n$, $| \cE | = \frac{2}{3}(n-1) + 1$, such that $\cH$ contains one $(5)$-hyperedge, one $(2,2,1)$-hyperedge and  the remaining hyperedges in $\cH$ are $(3,2)$-hyperedges.
\end{lemma}

{\em To the proof.}
Let
$\mathcal{H}=(V, \mathcal{E})$ with
$V=\{1,2,...,n\}$ and \\[1ex]
$\mathcal{E}=\{e_i=\{i,i+1,i+2,i+8,i+9\} \text{ }| \text{ }i \in \{1,4,7,...,n-15\}\} \quad \cup $\\[0.5ex]
\phantom{$\mathcal{E}=$ }$ \{e_i=\{i,i+1,i+2,i+5,i+6\} \text{ }| \text{ }i \in \{2,5,8,...,	
n-14\}\} \quad \cup $ \\[0.5ex]
\phantom{$\mathcal{E}=$ }$\{\{n-12,n-11,n-10,n-6,n-5\}, $ \\[0.5ex]
\phantom{$\mathcal{E}=$ }$\hspace{2,15mm}\{n-11,n-10,n-9,n-4,n-3\}, $ \\[0.5ex]
\phantom{$\mathcal{E}=$ }$\hspace{2,15mm}\{n-9,n-8,n-7,n-1,n\}, $ \\[0.5ex]
\phantom{$\mathcal{E}=$ }$\hspace{2,15mm}\{n-8,n-7,n-6,n-3,n-2\},$ \\[0.5ex]
\phantom{$\mathcal{E}=$ }$\hspace{2,15mm}\{n-6,n-5,n-4,n,1\},$ \\[0.5ex]
\phantom{$\mathcal{E}=$ }$\hspace{2,15mm}\{n-5,n-4,n-3,3,4\},$ \\[0.5ex]
\phantom{$\mathcal{E}=$ }$\hspace{2,15mm}\{n-3,n-2,n-1,6,7\},$ \\[0.5ex]
\phantom{$\mathcal{E}=$ }$\hspace{2,15mm}\{n-2,n-1,n,1,2\},$ \\[0.5ex]
\phantom{$\mathcal{E}=$ }$\hspace{2,15mm}\{4,5,9,n-1,n\}\}$.

\smallskip
\noindent
Note that -- in comparison to Lemma \ref{Lemma_(3,2)_0_mod_3} -- the nine extra hyperedges require a much more voluminous case distinction (e.g. in \cite{Pae} on a scale of  about 6 pages).\hqed

\medskip
{\em Example.} For $n = 19$ we have $\cH = ( V, \cE)$ with $V = \{ 1, 2, \ldots, 19 \}$ and \\[1ex]
$\mathcal{E}= \{ \{1,2,3,9,10\},\{4,5,6,12,13\},$ \\[0.5ex]
\phantom{$\mathcal{E}=$ }$\hspace{2,15mm}\{2,3,4,7,8\},\{5,6,7,10,11\},$ \\[0.5ex]
\phantom{$\mathcal{E}=$ }$\hspace{2,15mm}\{7,8,9,13,14\},\{8,9,10,15,16\},\{10,11,12,18,19\},\{11,12,13,16,17\},
\{13,14,15,19,1\},$ \\[0.5ex]
\phantom{$\mathcal{E}=$ }$\hspace{2,15mm}\{14,15,16,3,4\},\{16,17,18,6,7\},\{17,18,19,1,2\},
\{4,5,9,18,19\}\}$.

\medskip
\noindent
In Figure 8, the (5)-hyperedge $\{ 17, 18, 19, 1, 2 \}$ and the (2,2,1)-hyperedge $\{4,5,9,18,19\}$ is drawn dark green and light blue, respectively.
\begin{figure}[h]
	\centering
	\includegraphics[width=11.5cm]{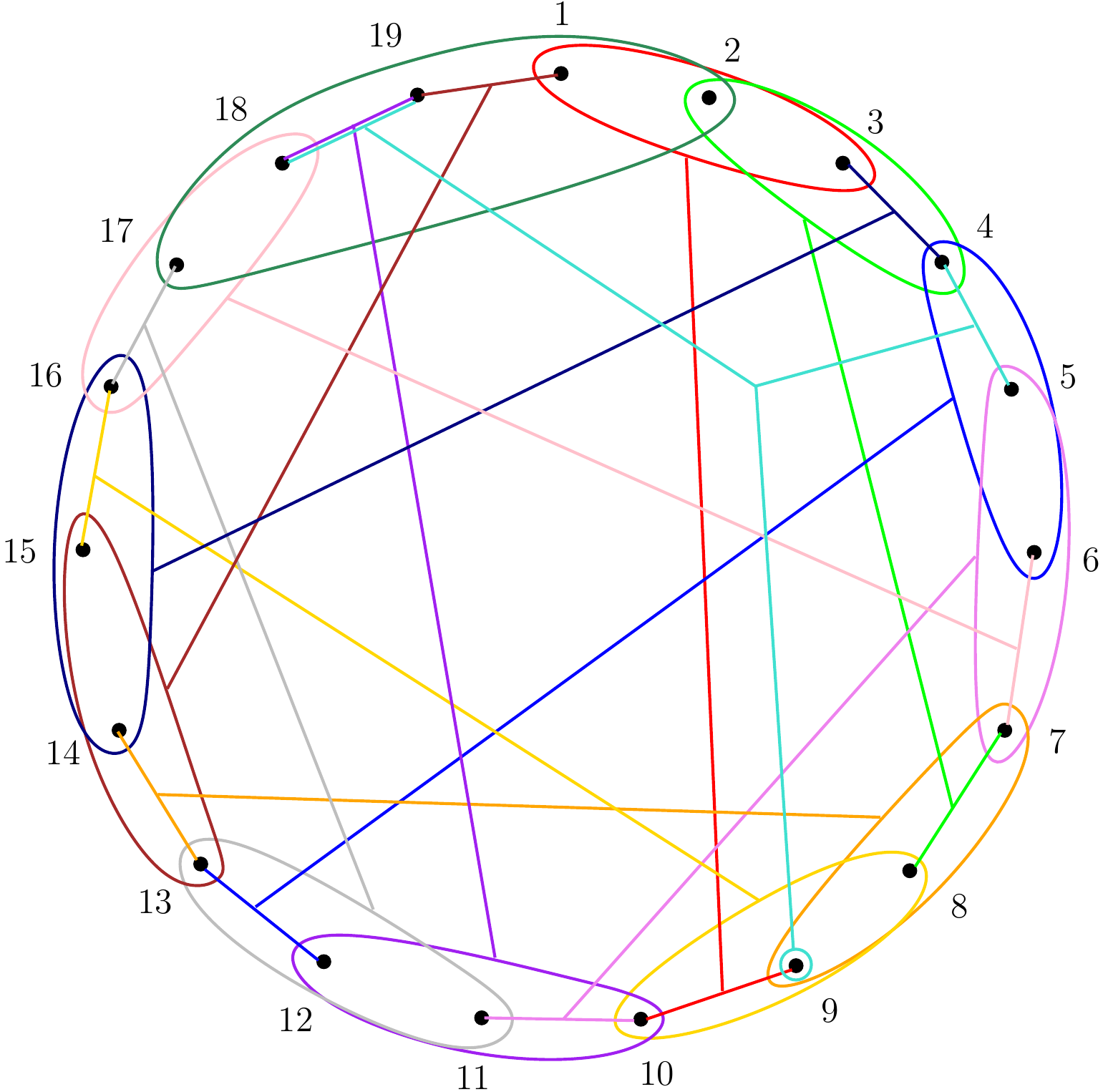}
	\caption{A $5$-uniform hypergraph $\mathcal{H}=(V,\mathcal{E})$ with $EI(\mathcal{H})=C_{19}$.}
	\label{img:Bsp19Knoten}
\end{figure}

\begin{lemma}[\cite{Pae}]
\label{Lemma_(3,2)_2_mod_3}
Let $n \ge 20$ and $n \equiv 2 \mod 3$. Then there exists a $5$-uniform hypergraph $\mathcal{H}=(V,\mathcal{E})$ with $EI(\mathcal{H})=C_n$, $| \cE | = \frac{2}{3}(n-2) + 1$, such that $\cH$ contains exactly one $(5)$-hyperedge and  the remaining hyperedges in $\cH$ are $(3,2)$-hyperedges.
\end{lemma}

{\em To the proof.}
We take
$\mathcal{H}=(V, \mathcal{E})$ with
$V=\{1,2,...,n\}$ and \\[1ex]
$\mathcal{E}= \{e_i=\{i,i+1,i+2,i+8,i+9\} \text{ }| \text{ }i \in \{1,4,7,...,n-16\}\}\quad \cup $\\[0.5ex]
\phantom{$\mathcal{E}=$ }$\{e_i=\{i,i+1,i+2,i+5,i+6\} \text{ }| \text{ }i \in \{2,5,8,...,	
n-15\}\}\quad \cup $\\[0.5ex]
\phantom{$\mathcal{E}=$ }$\{\{n-13,n-12,n-11,n-7,n-6\},$\\[0.5ex]
\phantom{$\mathcal{E}=$ }$\hspace{2,15mm}\{n-12,n-11,n-10,n-5,n-4\},$\\[0.5ex]
\phantom{$\mathcal{E}=$ }$\hspace{2,15mm}\{n-10,n-9,n-8,n-1,n\},$\\[0.5ex]
\phantom{$\mathcal{E}=$ }$\hspace{2,15mm}\{n-9,n-8,n-7,n-4,n-3\},$\\[0.5ex]
\phantom{$\mathcal{E}=$ }$\hspace{2,15mm}\{n-7,n-6,n-5,n,1\},$\\[0.5ex]
\phantom{$\mathcal{E}=$ }$\hspace{2,15mm}\{n-6,n-5,n-4,3,4\},$\\[0.5ex]
\phantom{$\mathcal{E}=$ }$\hspace{2,15mm}\{n-4,n-3,n-2,6,7\},$\\[0.5ex]
\phantom{$\mathcal{E}=$ }$\hspace{2,15mm}\{n-3,n-2,n-1,4,5\},$\\[0.5ex]
\phantom{$\mathcal{E}=$ }$\hspace{2,15mm}\{n-2,n-1,n,1,2\}\}$.

\medskip
\noindent
Again, the scale of the case distinction  is  about 6 pages in \cite{Pae}.\hqed

\medskip
{\em Example.} For $n = 20$ we have $\cH = ( V, \cE)$ with $V = \{ 1, 2, \ldots, 20 \}$ and \\[1ex]
$\mathcal{E}=\{\{1,2,3,9,10\},\{4,5,6,12,13\},$ \\[0.5ex]
\phantom{$\mathcal{E}=$ }$\hspace{2,15mm}\{2,3,4,7,8\},\{5,6,7,10,11\},$ \\[0.5ex]
\phantom{$\mathcal{E}=$ }$\hspace{2,15mm}\{7,8,9,13,14\},\{8,9,10,15,16\},\{10,11,12,19,20\},\{11,12,13,16,17\},
\{13,14,15,20,1\},$ \\[0.5ex]
\phantom{$\mathcal{E}=$ }$\hspace{2,15mm}\{14,15,16,3,4\},\{16,17,18,6,7\},\{17,18,19,4,5\},
\{18,19,20,1,2\}\}.$

\medskip
\noindent
In Figure 9, the (5)-hyperedge $\{18,19,20,1,2\}$ is drawn light blue.
\begin{figure}[h]
	\centering
	\includegraphics[width=12cm]{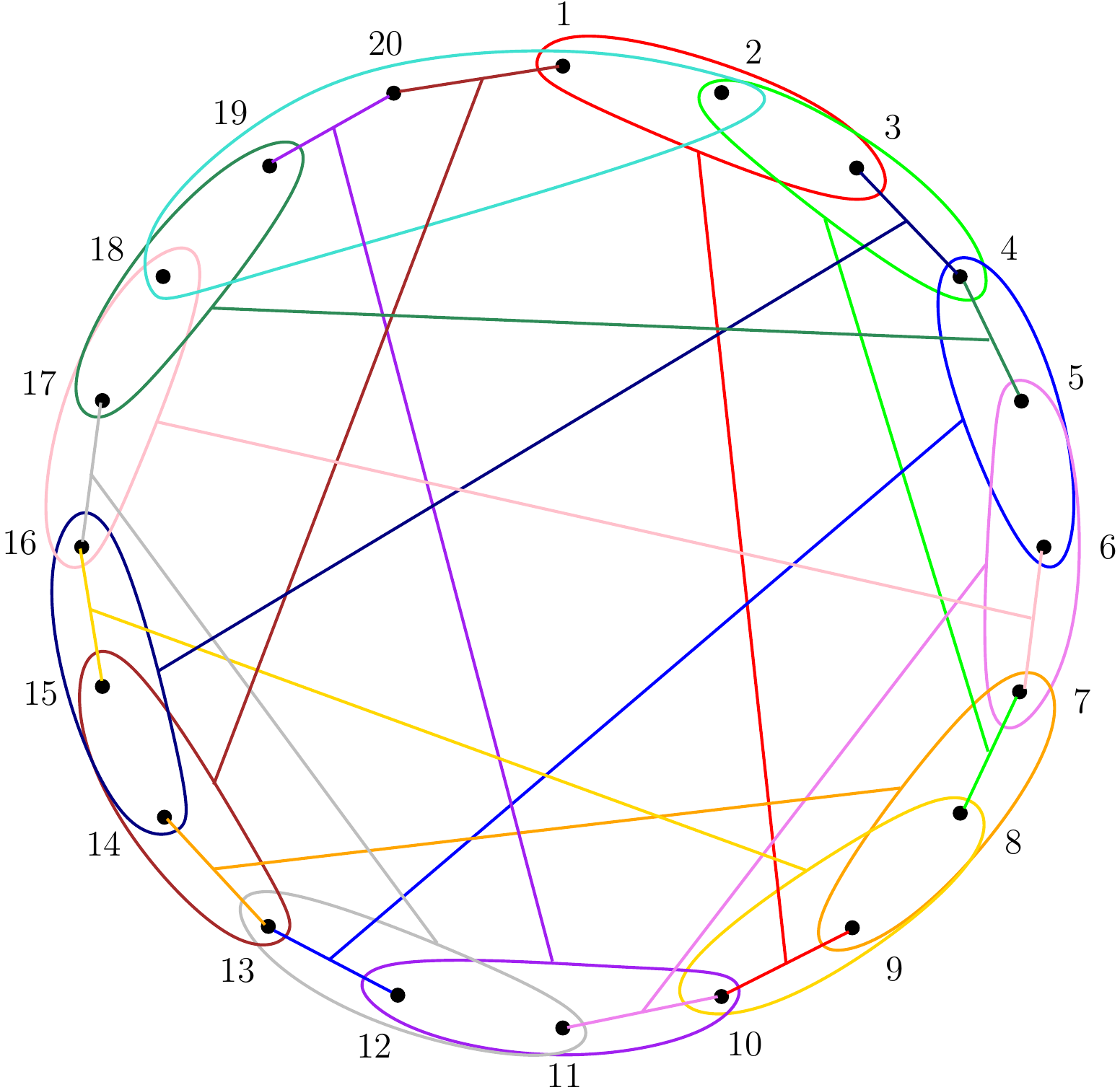}
	\caption{A $5$-uniform hypergraph $\mathcal{H}=(V,\mathcal{E})$ with $EI(\mathcal{H})=C_{20}$.}
	\label{img:Bsp20Knoten}
\end{figure}

To sum up the above Lemmas, we formulate the following theorem.

\begin{theorem}
\label{Theorem_(3,2)}
Let $n \ge 18$. Then there exists a $5$-uniform hypergraph $\mathcal{H}=(V,\mathcal{E})$ with $EI(\mathcal{H})=C_n$ and $| \cE | = \lceil \frac{2 n}{3} \rceil$ such that all but at most two of the hyperedges in $\cH$ are $(3,2)$-hyperedges.
\end{theorem}

\subsection{ $\mu^5_n = \lceil \frac{3 n}{5} \rceil$ for $n \equiv 0 \mod 5$}

With respect to Corollary \ref{CorMIN_l_uniform} the question arises, whether or not the lower bound $\lceil \frac{3 n}{5} \rceil$ for $| \cE|$  given in this Corollary is sharp. Because of $\lceil \frac{3 n}{5} \rceil < \lceil \frac{2 n}{3} \rceil$ and Theorem \ref{TheoMIN_(3,2)}, a greater variety of hyperedges than (3,2)-hyperedges has to be used. For $n \equiv 0 \mod 5$ we will show that the bound $\lceil \frac{3 n}{5} \rceil = \frac{3 n}{5}$ is a sharp one, i.e. in this case $\mu^5_n = \lceil \frac{3 n}{5} \rceil$. If $n$ is not a multiple of 5, this problem is still open (see Conjecture \ref{Conj}).

\smallskip
In 5-uniform hypergraphs, there can be seven possible types of hyperedges: (5)-, (4,1)-, (3,2)-, (3,1,1)-, (2,2,1)-, (2,1,1,1)- and (1,1,1,1,1)-hyperedges. Note that the (5)-hyperedges in $\cH$ can provide four half-edges in $EI(\mathcal{H})=C_n$. Three half-edges  can result from any (4,1)- and (3,2)-hyperedge, whereas the number of half-edges generated by (3,1,1)- and (2,2,1)-hyperedges is two. (2,1,1,1)-hyperedges lead only to one half-edge.  Obviously, in order to minimize $|\cE|$ in $ \cH = ( V, \cE)$ with $EI(\mathcal{H})=C_n$, the (1,1,1,1,1)-hyperedges make no sense. But which combinations of the other six types are good candidates for a minimum  $|\cE|$?

In \cite{Pae} an answer is given by proving Corollary \ref{CorMIN_l_uniform} in an alternative way, namely by means of linear optimization. For $n \equiv 0 \mod 5$, this leads to a construction which shows  $\mu^5_n = \lceil \frac{3 n}{5} \rceil$.
We give a rough description of this approach.

Let $\cH = ( V, \cE)$ be a 5-uniform hypergraph with $n$ vertices containing no (1,1,1,1,1)-hyperedge. We introduce six variables $x_5, x_4, x_{32}, x_3, x_{22}$ and $x_2$ corresponding to the possible six types of hyperedges in $\cE$, i.e. to the (5)-, (4,1)-, (3,2)-, (3,1,1)-, (2,2,1)- and (2,1,1,1)-hyperedges, respectively.
%, in the following sense.
Of course, the cardinality $| \cE|$ is the sum of the numbers of the hyperedges of each type. Using the described variables, this sum shall be written as

\centerline{$| \cE | = x_5 \cdot n + x_4 \cdot n +  x_{32} \cdot n +  x_3 \cdot n +  x_{22} \cdot n + x_2  \cdot n$.}

So, depending on $n$, our variables describe how many hyperedges of the corresponding type (represented by the indices of the variables) are contained in $\cE$.

Because we are searching for a lower bound for the number of hyperedges of a $5$-uniform hypergraph $\mathcal{H}=(V,\mathcal{E})$ with $EI(\mathcal{H})=C_n$, this leads to the following objective function which is to be minimized.

\begin{equation}
 x_5 \cdot n + x_4 \cdot n +  x_{32} \cdot n +  x_3 \cdot n +  x_{22} \cdot n + x_2  \cdot n \; \rightarrow \; min .
 \end{equation}

\medskip
The non-negativity of the variables has to be guaranteed:

\begin{equation}
x_5, x_4, x_{32}, x_3, x_{22}, x_2 \ge 0.
 \end{equation}

\medskip

To generate the $n$ edges of $C_n$ we need at least $2n$ half-edges. In connection with the possible number of half-edges of each of the six types of hyperedges this yields the constraint

\begin{equation}
 2n \le 4  x_5 \cdot n + 3 x_4 \cdot n +  3 x_{32} \cdot n +  2 x_3 \cdot n +  2 x_{22} \cdot n + x_2  \cdot n .
 \end{equation}

\medskip

At the beginning of Subsection 4.1 we shortly discussed the problem of "middle vertices" in a 5-section $(i, i+1, i+2, i+3, i+4)$ and a 4-section $(i, i+1, i+2, i+3)$, respectively. That is -- in case of a 5-section $(i, i+1, i+2, i+3, i+4)$ in $\cH$ -- to obtain the edges $\{i+1, i+2\} \in E(C_n)$ and $\{i+2, i+3\} \in E(C_n)$, we need two 2-sections $(i+1, i+2)$ and $(i+2, i+3)$ in $\cH$ which have to come from some $(3,2)$-, $(2,2,1)$ or $(2,1,1,1)$-hyperedge. So we have to have "sufficiently many" such 2-sections in our hypergraph $\cH$. Analogously, we have to argue in case of a 4-section $(i, i+1, i+2, i+3)$ in $\cH$ -- in order to obtain the edge $\{i+1, i+2\} \in E(C_n)$ we need one 2-section, namely $(i+1, i+2)$.

This leads to the constraint

\begin{equation}
 2 x_5 \cdot n +  x_4 \cdot n  \le  x_{32} \cdot n +    2 x_{22} \cdot n + x_2  \cdot n.
 \end{equation}

\medskip
We divide (1), (3) and (4) by $n$ and obtain the linear optimization problem

\begin{align}\label{opti}
x_{5}+x_{4}+x_{32}+x_{3}+x_{22}+x_{2} &\rightarrow min\\
2&\leq 4x_{5}+3x_{4}+3x_{32}+2x_{3}+2x_{22}+x_{2}  \notag  \\
2x_{5}+x_{4}&\leq x_{32}+2x_{22}+x_{2}\notag \\
x_{5},x_{4},x_{32},x_{3},x_{22},x_{2}&\geq 0 \notag
\end{align}.

In \cite{Pae}, a generalization of the Simplex algorithm (the {\em Big M method}, see \cite{Dom}) is used to solve (5). The solution of (5) includes two results. The first one is the lower bound $|\cE| \ge \frac{3}{5}n$ for the number of hyperedges in $\cH$. As a second result we obtain the values $\frac{2}{5}$ and $\frac{1}{5}$ for the two basis variables $x_{32}$ and $x_5$, respectively.

In order to reach the lower bound $\frac{3}{5}n$ for the number of hyperedges, $n$ has to be a multiple of 5. So let $n \equiv 0 \mod 5$. Since the remaining variables $x_4, x_3, x_{22}, x_2 $ play no role in the solution, the idea suggests itself to try to construct a 5-uniform hypergraph $\cH = (V, \cE)$ with $EI(\mathcal{H})=C_n$ that contains $\frac{2}{5} n$ (3,2)-hyperedges and $\frac{1}{5} n$ (5)-hyperedges. This idea had been very useful to prove the next theorem.

\begin{theorem}[\cite{Pae}]
\label{Theorem_Min5}
Let $n \ge 20$ and $n \equiv 0 \mod 5$. Then there exists a 5-uniform hypergraph $\mathcal{H}=(V,\mathcal{E})$ with
$EI(\mathcal{H})=C_n$ and $| \cE | = \frac{3}{5} n$.
\end{theorem}

{\em To the proof.}
The hypergraph $\mathcal{H}=(V,\mathcal{E})$ with $V=\{1,2,...,n\}$ and
\begin{flalign*}
\mathcal{E}=&\{e_i=\{i,i+1,i+2,i+3,i+4\} \text{ }| \text{ }i \in \{1,6,11,...,n-4\}\} &\cup &&&&&&&&\phantom{h}\\
 &\{e_i=\{i,i+1,i+2,i-6,i-5\} \text{ }| \text{ }i \in \{4,9,14,...,	
n-1\} \}&\cup\\
&\{e_i=\{i,i+1,i+2,i+7,i+8\} \text{ }| \text{ }i \in \{5,10,15,...,n\} \}
\end{flalign*}
has the required properties. \\
Again, the proof can be done by a detailed case distinction (see \cite{Pae}). \hqed

\medskip
{\em Example.} For $n = 20$ we have $\cH = ( V, \cE)$ with $V = \{ 1, 2, \ldots, 20 \}$ and
\begin{flalign*}
\mathcal{E}=\{&\{1,2,3,4,5\},\{6,7,8,9,10\},\{11,12,13,14,15\},\{16,17,18,19,20\},\hspace{4,2cm}\\
&\{4,5,6,18,19\},\{9,10,11,3,4\},\{14,15,16,8,9\},\{19,20,1,13,14\},\\
&\{5,6,7,12,13\},\{10,11,12,17,18\},\{15,16,17,2,3\},\{20,1,2,7,8\}\}.
\end{flalign*}

\medskip
\noindent
$\cH = ( V, \cE)$ contains $\frac{3}{5} n = 12$ hyperedges and its edge intersection hypergraph is the cycle $C_{20}$; the hypergraph  is drawn in Figure 10.
\begin{figure}[h]
	\centering
	\includegraphics[width=12cm]{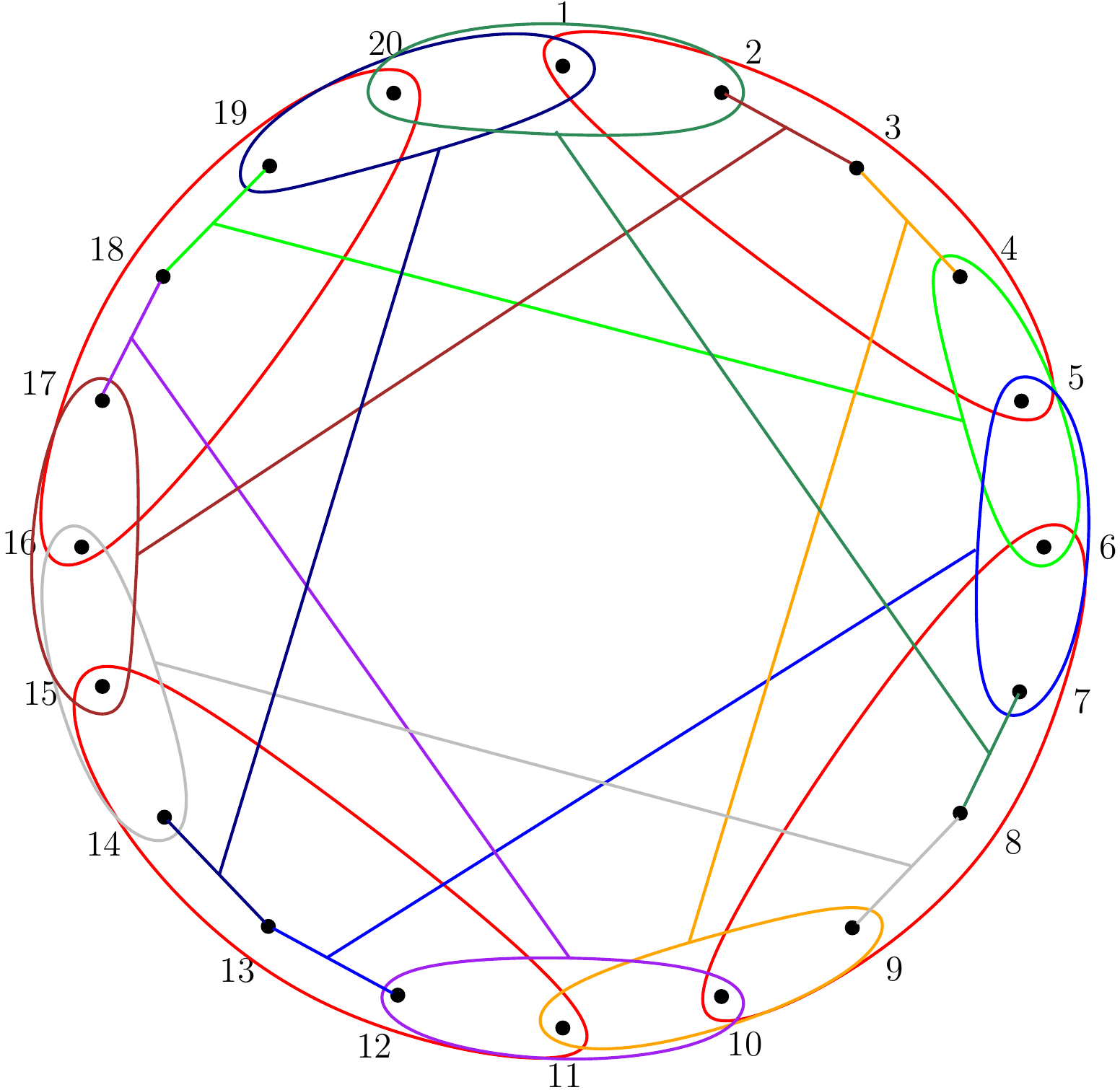}
	\caption{A $5$-uniform hypergraph $\mathcal{H}=(V,\mathcal{E})$ with $EI(\mathcal{H})=C_{20}$ and $| \cE | = 12$.}
	\label{img:Bsp20Knoten5}
\end{figure}

\section{Concluding remarks}

For $ k \le 6$ and sufficiently large $n$, we gave constructions of $k$-uniform hypergraphs $\cH$ having the edge intersection hypergraph $EI(\mathcal{H})=C_{n}$. The constructed hypergraphs $\cH = ( V, \cE)$ have minimum cardinality of the set of hyperedges, i.e.   $| \cE | = \mu^k_n$, if $k \in \{3,4,6 \}$ and $k = 5 \; \wedge \; n \equiv 0 \mod 5$, respectively.

Consequently, in theses cases, the ceiling of the lower bound $\frac{3n}{k}$ for $\mu^k_n$ given in Corollary \ref{CorMIN_l_uniform} is a sharp bound for $\mu^k_n$, i.e. $\mu^k_n = \lceil \frac{3n}{k} \rceil$.
This leads to the following conjecture.

\begin{conjecture}
\label{Conj}
For any $k \ge 3$ there exists an $n_k \in \N$ such that for every $n \ge n_k$ we have $\mu^k_n = \lceil \frac{3n}{k} \rceil$.
\end{conjecture}

Note that this Conjecture is open not only for all $k > 6$ but also for $k = 5 \; \wedge \; n \not\equiv 0 \mod 5$.


\begin{thebibliography}{}	
\bibitem{GST5} C. Berge, \textit{Graphs and Hypergraphs}, North Holland, Amsterdam (1973).	
%\bibitem{Bie} Th. Biedl, M. Stern, \textit{On edge-intersection graphs of $k$-bend paths
%in grids}, Discr.  Math. and Theor. Comp. Sci. \textbf{12(1)} (2010) 1--12.
%\bibitem{Cam} K. Cameron, S. Chaplick, C.T. Hoang, \textit{Edge intersection graphs of $L$-shaped paths in grids}, Discr. Appl. Math. \textbf{210} (2016) 185--194.
%%\bibitem{GST2016}  C. Garske, M. Sonntag, H.-M. Teichert, \textit{Niche hypergraphs}, Discussiones Mathematicae Graph Theory \textbf{36} (2016) 819--832.
%\bibitem{Gol} M.C. Golumbic, R.E. Jamison, \textit{The Edge Intersection Graphs
%of Paths in a Tree}, J. Comb. Theory, Series B \textbf{38} (1985) 8--22.
\bibitem{Dom} W. Domschke, A. Drexl, R. Klein, A. Scholl, \textit{Einführung in Operations Research}, Springer Gabler (2015).
\bibitem{Pae} S. P\"atz, \textit{Konstruktion uniformer Hypergraphen zur Erzeugung von Kreisen als Edge-Intersection-Hypergraph}, diploma thesis, TU Bergakademie Freiberg (2020).
%\bibitem{GST33} J. Park, Y. Sano, \textit{The double competition hypergraph of a digraph}, Discr. Appl. Math. \textbf{195} (2015) 110 --113.
%\bibitem{RW} R.C. Read, R.J. Wilson, \textit{An Atlas of Graphs}, Oxford University Press, New York, 1998.
%\bibitem{ROB} G. Robers, \textit{EI-Hypergraphen und deren Eigenschaften}, Bachelor Thesis, University of L\"ubeck, 2017.
%\bibitem{Sku} P.V. Skums, S.V. Suzdal, R.I. Tyshkevich, \textit{Edge intersection graphs of linear 3-uniform hypergraphs}, Discr. Math. \textbf{309} (2009) 3500--3517.
\bibitem{STEIH} M. Sonntag, H.-M. Teichert, \textit{Edge intersection hypergraphs},  Discussiones Mathematicae Graph Theory (2021, in press), 1-26, DOI: 10.7151/dmgt.2435.
%\bibitem{STCac} M. Sonntag, H.-M. Teichert, \textit{Nearly all cacti are edge intersection hypergraphs of 3-uniform hypergraphs}, arXiv:1906.05639 [math.CO] (2019), 1-9.
\bibitem{ST} M. Sonntag, H.-M. Teichert, \textit{Cycles as edge intersection hypergraphs}, arXiv:1902.00396 [math.CO] (2019), 1-21.
\bibitem{WMath} Wolfram Research, Inc., MATHEMATICA$^{\textregistered}$, Version 8.0, Champaign, IL (2010).
\end{thebibliography}
\end{document}